\documentclass[12pt]{amsart}

\usepackage{hyperref}
\usepackage{color}
\usepackage{pinlabel}
\usepackage{enumerate,enumitem}
\usepackage{xypic}
\usepackage{graphicx}
\usepackage{booktabs}

\usepackage{amsmath,amsfonts,amssymb}
\usepackage{psfrag}

\theoremstyle{plain}
\newtheorem*{theorem*}{Theorem}
\newtheorem*{lemma*}{Lemma}
\newtheorem*{corollary*}{Corollary}
\newtheorem*{corollary-1}{Corollary 1}
\newtheorem*{corollary-2}{Corollary 2}

\newtheorem*{proposition*}{Proposition}
\newtheorem*{proposition-1}{Proposition 1}
\newtheorem*{proposition-2}{Proposition 2}
\newtheorem*{proposition-3}{Proposition 3}

\newtheorem{conjecture*}{Conjecture}
\newtheorem{theorem}{Theorem}[section]
\newtheorem{lemma}[theorem]{Lemma}

\newtheorem{proposition}[theorem]{Proposition}

\newtheorem{conjecture}[theorem]{Conjecture}
\newtheorem{question}[theorem]{Question}

\theoremstyle{remark}

\newtheorem*{definition}{Definition}

\newtheorem*{claim}{Claim}

\theoremstyle{definition}

\def\op{\operatorname}

\def\gl{\operatorname{GL}}    \def\Z{\mathbb{Z}}   
\def\N{\mathbb{N}}  \def\l{\lambda}  
   \def\tor{\operatorname{Tor}} \def\bp{\begin{pmatrix}}
 \def\ep{\end{pmatrix}} 
\def\bn{\begin{enumerate}} 

   \def\en{\end{enumerate}}
\def\ba{\begin{array}} \def\ea{\end{array}}      \def\ti{\tilde} \def\wti{\widetilde}
\def\id{\operatorname{id}}   
  
\def\ker{\operatorname{ker}}\def\be{\begin{equation}} \def\ee{\end{equation}}

\def\ol{\overline}

\def\ti{\tilde}

\def\mod{\op{Mod}}

\def\tor{\op{Tor}}

\newcommand{\smfrac}[2]{\mbox{\footnotesize$\displaystyle\frac{#1}{#2}$}}

\newcommand{\ttmprod}[2]{\mbox{\footnotesize{$\textstyle \prod\limits_{#1}^{#2}$}}}

\renewcommand\epsilon{\varepsilon}
\DeclareMathAlphabet{\mathbf}{OML}{cmm}{b}{it}

\numberwithin{equation}{section}

\makeindex

\begin{document}

\title{Torsion in the homology of finite covers of  3-manifolds}

\author{Stefan Friedl}
\address{Fakult\"at f\"ur Mathematik\\ Universit\"at Regensburg\\   Germany}
\email{sfriedl@gmail.com}

\author{Gerrit Herrmann}
\address{Fakult\"at f\"ur Mathematik\\ Universit\"at Regensburg\\   Germany}
\email{gerrit.herrmann@mathematik.uni-regensburg.de}

\begin{abstract}
Let  $N$ be a prime 3-manifold that is not a closed graph manifold. Building on a result of Hongbin Sun  and using a result of Asaf Hadari we show that  for every $k\in\N$ there exists a finite cover $\wti{N}$ of $N$ such that $|\tor H_1(\wti{N};\Z)|>k$. 
\end{abstract}

\maketitle

\section{Introduction}
Let $N$ be a 3-manifold. Here and throughout the paper all 3-manifolds are understood to be compact, orientable, connected and with empty or toroidal boundary.
If $N$ is prime, then we denote by $\op{vol}(N)$ the sum of the volumes of the hyperbolic pieces in the geometric decomposition of $N$.
Furthermore, given an abelian group $A$ we denote by $\tor A$ the torsion subgroup.
The following conjecture has been formulated in  similar forms by many mathematicians.  (See e.g.~\cite{BV13}, ~\cite[Conjecture~1.12]{Lu13}, ~\cite[Conjecture~1]{Le14} and~\cite[Question~7.5.2]{AFW15})

\begin{conjecture}\label{conj:vol-tor}
For every prime 3-manifold $N$  there exists a cofinal regular tower $\{\widetilde{N}_n\}$ of $N$
such that
\[  \lim_{n\to\infty}  \smfrac{1}{[\ti{N}_n:N]} \ln\big| \operatorname{Tor} H_1(\widetilde{N}_n;\Z)\big|\,\,=\,\,\smfrac{1}{6\pi} \operatorname{vol}(N)~?\]
\end{conjecture}

For graph manifolds, i.e.\ manifolds with $\op{vol}(N)=0$, this conjecture, in fact a stronger version of it, was independently proved by Le~\cite{Le14} and Meumertz\-heim~\cite{Me16}.

A proof of Conjecture~\ref{conj:vol-tor} would in particular imply that if $\op{vol}(N)$ is greater than zero, i.e.\ if $N$ is not a graph manifold, then $N$ admits  finite covers such that the torsion in the first homology becomes arbitrarily large.

Even this statement is highly non-trivial. For closed hyperbolic 3-manifolds this was first proved by Sun \cite{Su15}, building on the techniques pioneered by Kahn--Markovic \cite{KaM12}. More precisely, Sun proved the following theorem.

\begin{theorem}\label{thm:sun}\textbf{\emph{(Sun)}}
If $N$ is a closed hyperbolic 3-manifold, then for every $k\in\N$ there exists a finite cover $\wti{N}$ of $N$ $($not necessarily regular$)$ such that $|\tor H_1(\wti{N};\Z)|>k$.
\end{theorem}

One can ask for which other 3-manifolds does the conclusion of Theorem~\ref{thm:sun} hold. Evidently the conclusion does not hold for some Seifert fibered spaces like $S^3$ or products $S^1\times \Sigma$. 

In this short note we observe that a recent theorem of Hadari~\cite{Ha15}, stated below as Theorem~\ref{thm:hadari}, can be used to extend the conclusion of Theorem~\ref{thm:sun} to  all hyperbolic 3-manifolds with boundary and more generally, to all  prime 3-manifolds  that are not  graph manifolds.

\begin{theorem}\label{mainthm}
If $N$ is a prime 3-manifold that is not a graph manifold, then for every $k\in\N$ there exists a finite cover $\wti{N}$ of $N$ such that $|\tor H_1(\wti{N};\Z)|>k$.
\end{theorem}

In many cases the finite covers we produce will turn out to be irregular covers. \\

Using Hadari's theorem we can also prove the following result on virtual torsion homology of graph manifolds.

\begin{theorem}\label{mainthm2}
Let $N$ be a virtually fibered graph manifold, which is not a Seifert fibered space. Then for every $k\in\N$ there exists a finite cover $\wti{N}$ of $N$ such that $|\tor H_1(\wti{N};\Z)|>k$.
\end{theorem}

Note that by \cite[Theorem~0.1]{WY97} every graph manifold with non-trivial boundary is virtually fibered. So the statement of Theorem~\ref{mainthm2} applies in particular to any graph manifold with boundary that is not a Seifert fibered space. We conclude this introduction with the following question.

\begin{question}
Does every graph manifold with non-trival JSJ decomposition admit arbitrarily large torsion homology in finite covers?
\end{question}

\noindent \emph{Final remark.} Just as we were about to  finish this paper we saw a preprint by Yi Liu~\cite{Li17}. He showed in particular, see \cite[Corollary~1.4]{Li17}, that every irreducible compact 3-manifold with empty or toroidal boundary that is not a graph manifold admits a \emph{regular}
finite cover $\wti{N}$ such that $|\tor H_1(\wti{N};\Z)|\ne 0$.

\subsection*{Conventions.} All manifolds, in particular all surfaces, are assumed to be connected, compact and orientable, unless we say explicitly otherwise.

\subsection*{Acknowledgment.}
Both authors gratefully acknowledge the support provided by the SFB 1085 `Higher
Invariants' at the University of Regensburg, funded by the DFG.

\section{Proofs}

\subsection{Preparations}

Given a  surface $S$ we denote by $\mod(S)$ the mapping class group of $S$, i.e.\ the group of all self-diffeomorphisms of $S$ up to isotopy. The following theorem, that extends earlier work of Koberda and Mangahas~\cite{Ko15,KoM15},  was recently proved by Hadari~\cite{Ha15}.

\begin{theorem}\label{thm:hadari}\textbf{\emph{(Hadari)}}
Let $S$ be a   surface with non-empty boundary. Let $\psi\in \mod(S)$ be an element of infinite order. Then there
exists a finite solvable cover $\wti{S}\to S$ and a  lift $\wti{\psi}$ of $\psi$ to $\wti{S}$ such that the action of
$\wti{\psi}$ on $H_1(\wti{S};\Z)$  has infinite order.
\end{theorem}
%
%

\begin{definition}
Let $S$ be a subsurface of a surface  $F$.
\bn
\item  We say $S$  is \emph{$\pi_1$-injective} if $\pi_1(S)\to \pi_1(F)$ is a monomorphism. (Recall that by our convention surfaces are connected.) Similarly we define \emph{$H_1$-injective}.
\item We say $S$ is a \emph{proper subsurface} if $S\ne F$.
\en
\end{definition}

Later on we will need the following  extension of Theorem~\ref{thm:hadari}.

\begin{proposition}\label{prop:hadari2}
Let $F$ be a  surface and let $\psi$ be a self-diffeomorphism of $F$.
If there exists a  $\pi_1$-injective proper subsurface $S$ of $F$ such that 
$\psi$ restricts to a self-diffeomorphism of $S$ with infinite order, then there
exists a finite  cover $\wti{F}\to F$ and a  lift $\wti{\theta}$ of some power $\theta=\psi^k$ such that the action of
$\wti{\theta}$ on $H_1(\wti{F};\Z)$  has infinite order.
\end{proposition}

In the proof of Proposition~\ref{prop:hadari2} we will need the following three lemmas, all three of them are surely known to the experts.

\begin{lemma}\label{lem:virtually-h1-injective}
Let $F$ be a surface and let $S$ be a subsurface of $F$ that is $\pi_1$-injective. Let $x_0\in S$ be a base point.  There exists a finite-index subgroup $\ol{\pi}$ of $\pi_1(F,x_0)$ that contains $\pi_1(S,x_0)$ and such that $\pi_1(S,x_0)\subset \ol{\pi}$ is a retract.
\end{lemma}

\begin{proof}
This lemma was shown implicitly by Scott~\cite{Sc78,Sc85}, see also  \cite[Theorem~18]{FW16}.
\end{proof}

\begin{lemma}\label{lem:power-of-psi-lifts}
Let $F$ be a surface and let $\psi\colon F\to F$ be a diffeomorphism that fixes a base point $x_0$. Let $p\colon \ol{F}\to F$ be a finite cover and let $\ol{x}_0$ be a pre-image of $x_0$. Then there exists a $k\in \N$ such that $\psi^k$ lifts to a diffeomorphism $\ol{\phi}\colon\ol{F}\to \ol{F}$ with
$\ol{\phi}(\ol{x}_0)=\ol{x}_0$.
\end{lemma}

\begin{proof}
We pick a base point $\ol{x}_0$ of $\ol{F}$ with $p(\ol{x}_0)=x_0$. 
We write $\pi=\pi_1(F,x_0)$ and we write $m=[F:\ol{F}]=[\pi_1(\ol{F},\ol{x}_0):\pi]$. 
We denote by $G_m(\pi)$ the set of subgroups of $\pi$ of index $m$.
The  isomorphism $\psi_*\colon \pi\to \pi$ acts on $G_m(\pi)$ in an obvious way. Note that the set $G_m(\pi)$ is finite since $\pi_1(F,x_0)$ is finitely generated.
It follows that there exists a $k\in \N$ such that the action of $\psi_*^k$ on $G_m(\pi)$ is trivial. But this implies in particular that $\psi_*^k$ preserves $p_*(\pi_1(\ol{F},\ol{x}_0))$, i.e.\ $\psi_*^k$ restricts to an isomorphism $p_*(\pi_1(\ol{F},\ol{x}_0))\to p_*(\pi_1(\ol{F},\ol{x}_0))$. Standard covering theory now implies that $\psi^k$ lifts to a diffeomorphism $\ol{F}\to \ol{F}$.
\end{proof}

\begin{lemma}\label{lem:lifts-of-homeomorphisms}
Let $p\colon \wti{X}\to X$ be an $m$-fold covering of a surface $X$. Let $x_0\in X$ be a base point and let $\wti{x}_0$ a point in $\wti{X}$ with $p(\wti{x}_0)=x_0$. Let $\psi\colon X\to X$ be a homeomorphism with $\psi(x_0)=x_0$. Then the following statements hold:
\bn[font=\normalfont]
\item If $\wti{\psi}$ is a lift of $\psi$ to a homeomorphism of $\wti{X}$, then $\wti{\psi}^{m!}(\wti{x}_0)=\wti{x}_0$.
\item Suppose $\alpha$ and $\beta$ are two lifts  of $\psi$ to a homeomorphism of $\wti{X}$. If $\alpha(\wti{x}_0)=\beta(\wti{x}_0)$, then $\alpha=\beta$.
\en
\end{lemma}

\begin{proof}
Since $\wti{\psi}$ is a lift it acts via permutation on the set $p^{-1}(x_0)$ and hence the first statement follows.
The second statement follows from the uniqueness of lifts.
\end{proof}

\begin{proof}[Proof of Proposition~\ref{prop:hadari2}]
Let $F$ be a  surface, let $\psi$ be a self-diffeomorphism of $F$
and let $S$ be a $\pi_1$-injective proper subsurface of $F$ such that 
$\psi$ restricts to a self-diffeomorphism of $S$ with infinite order. After an isotopy we can without loss of generality assume that $\psi$ fixes a base point $x_0\in S$. 

By Lemma~\ref{lem:virtually-h1-injective} there exists a  finite-index subgroup $\ol{\pi}$ of $\pi_1(F,x_0)$ that contains $\pi_1(S,x_0)$ and such that there exists a retraction $\gamma\colon  \ol{\pi}\to \pi_1(S,x_0)$.
Thus we obtain  the following commutative diagram
\[ \xymatrix@C1.5cm@R0.5cm{ 
&\ol{\pi}\ar@<-2.5pt>[dl]_\gamma \ar@{^(->}[d]\\
\pi_1(S,x_0)\ar@<-2.5pt>@{_(->}[r]\ar@{_(->}@<-2.5pt>[ur]&\pi_1(F,x_0).}\]
We denote by $q\colon \ol{F}\to F$ the finite covering corresponding to $\ol{\pi}$.  
By elementary covering theory we can and will view $S$ as a subsurface of $\ol{F}$. In particular we will equip $S$ and $\ol{F}$ with the base point $x_0\in S\subset \ol{F}$. 

Furthermore, by Lemma~\ref{lem:power-of-psi-lifts} there exists a $l\in \N$ such that $\psi^l$ lifts to a self-diffeomorphism $\ol{\phi}$ of $\ol{F}$ with $\ol{\phi}(x_0)=x_0$. Note that  $\ol{\phi}|_S=(\psi|_S)^l$.
In particular $\ol{\phi}|_S$ has infinite order. Therefore we obtain from Theorem~\ref{thm:hadari} a finite regular cover $r\colon \wti{S}\to S$ of degree $h$ such that $\sigma=\ol{\phi}|_S$ lifts to a self-diffeomorphism $\wti{\sigma}$ of $\wti{S}$ such that the induced action of $\wti{\sigma}$ on $H_1(\wti{S};\Z)$ has infinite order. 

We denote by $\varphi\colon \pi_1(S,x_0)\to G$ the epimorphism given by the projection map $\pi_1(S,x_0)\to \pi_1(S,x_0)/\pi_1(\wti{S},x_0)$. We denote by $r\colon \wti{F}\to \ol{F}$ the covering corresponding to $\ker(\varphi\circ \gamma)$. We can and will view $\wti{S}$ as a subsurface of $\wti{F}$. Summarizing we obtain the following diagrams
\[ \xymatrix@C1.5cm@R0.55cm{ 
&\pi_1(\wti{F},\wti{x}_0) \ar@<-2.5pt>[dl]_\gamma \ar@{^(->}[d]^{r_*}\\
\pi_1(\wti{S},\wti{x}_0)\ar@{^(->}[d]_{r_*}\ar@{_(->}@<-2.5pt>[ur]&\pi_1(\ol{F},x_0)\ar@<-2.5pt>[dl]_\gamma \\
\pi_1(S,x_0)\ar@{_(->}@<-2.5pt>[ur]\ar[d]^\varphi&\\
G}
\quad\quad \quad
 \xymatrix@C1.5cm@R0.55cm{ 
&\wti{F} \circlearrowleft_{\wti{\phi}} \ar@<-10pt>@{->}[d]^{r}\\
\hspace{0.7cm}{}_{\wti{\sigma}^n}\circlearrowright \wti{S}\ar@<11pt>[d]_{r}\ar@{^(->}[ur]\hspace{0.95cm}&\hspace{0.15cm}\ol{F}\circlearrowleft_{\ol{\phi}^n} \\
{}_{\sigma^n=\ol{\phi}^n|_S}\hspace{-0.15cm}\circlearrowright S\ar@{^(->}[ur]\hspace{0.95cm}&\\
}
\]

We pick a base point $\wti{x}_0$ of $\wti{S}$ such that $r(\wti{x}_0)=x_0$. Note that by Lemma~\ref{lem:power-of-psi-lifts} we obtain an $m\in\N$ such that
the self-diffeomorphism $\ol{\varphi}:=\ol{\phi}^m$ of $\ol{F}$ lifts to a self-diffeomorphism $\wti{\varphi}$ of $\wti{F}$. 
 We write $n=h!$. It follows from Lemma~\ref{lem:lifts-of-homeomorphisms} (1) that $\wti{\varphi}^n$ and $\wti{\sigma}^n$ both fix the base point $\wti{x}_0$.
Note that  $\wti{\varphi}^n|_{\wti{S}}$ and $\wti{\sigma}^n$ are both lifts of $(\ol{\varphi}|_S)^n=\sigma^{nm}$.
Since both fix $\wti{x}_0$ we obtain from Lemma~\ref{lem:lifts-of-homeomorphisms} (2) that  $\wti{\varphi}^n|_{\wti{S}}=\wti{\sigma}^n$.

Since $\pi_1(\wti{S},\wti{x}_0)$ is a retract of $\pi_1(\wti{F},\wti{x}_0)$ we see that 
 $H_1(\wti{S};\Z)$ is also a retract of $H_1(\wti{F};\Z)$, in particular
the map   $H_1(\wti{S};\Z)\to H_1(\wti{F};\Z)$ is a monomorphism. 
Summarizing we see that the restriction of $\wti{\phi}^n$ to 
the subgroup  $H_1(\wti{S};\Z)\subset H_1(\wti{F};\Z)$ equals $\wti{\sigma}_*^n$, hence it has has infinite order.
But  $\wti{\theta}=\wti{\varphi}^n$ is a lift of $\psi^{k}$ with $k=lmn$.
\end{proof}

Given a surface $F$ and  $\psi\in \mod(F)$ we denote by $M(F,\psi)$ the corresponding mapping torus. Furthermore, by a slight abuse of notation we denote the induced isomorphism $H_1(F;\Z)\to H_1(F;\Z)$ by $\psi$ as well.
Also, given a free abelian group $H$ and an automorphism $\psi$ of $H$ we denote the corresponding semidirect product by $\Z\ltimes_\psi H$. 
For future reference we record the following well-known lemma.

\begin{lemma}\label{lem:h1fibered}
Given a surface $F$ and  $\psi\in \mod(F)$ we have the following isomorphisms
\[ H_1(M(F,\psi);\Z)\,\cong \,H_1(\Z\ltimes_\psi H_1(F;\Z);\Z)\,\cong\, \Z\oplus H_1(F;\Z)/(\psi-\id)H_1(F;\Z). \]
\end{lemma}

Later on we will need the following proposition.

\begin{proposition}\label{prop:eigenvalue-root-of-unity}
Let $A\in \gl(k,\Z)$ be a matrix of infinite order such that every eigenvalue of $A$ is a root of unity. Then given any $m\in \N$ there exists an $n$ such that 
\[ \Z^k/(A^n-\id)\Z^k\]
has an element of order $m$. 
\end{proposition}

\begin{proof}
Since all  eigenvalues of $A$ are roots of unity we can,  possibly after going to a finite power of $A$, assume that all eigenvalues are in fact equal to $1$.

\begin{claim}
The matrix  $B:=A-\id$ is nilpotent, more precisely, $B^k=0$.
\end{claim}

Since $1$ is the only eigenvalue of $A$  there exists a rational matrix $Q$ such that $QAQ^{-1}=\id+C$ where $C=(c_{ij})$ is a rational  matrix with $c_{ij}=0$ for $i\geq j$. In particular $C^k=0$. 
But then
\[ QB^kQ^{-1}\,\,=\,\, Q(A-\id)^kQ^{-1}\,\,=\,\, (QAQ^{-1}-\id)^k\,\,=\,\,C^k\,\,=0.\]
This concludes the proof of the claim.

\begin{claim}
Given  $m\in \N$ there exists an $n$ such that all entries of 
\[ A^n-\id \]
are divisible by $m$. 
\end{claim}

By the previous claim we have $B^n=0$ for all $n\geq k$. Now for any $n\geq k$ we have
\[  A^n-\id\,\,=\,\,(\id+B)^n-\id\,\,=\,\, \sum\limits_{i=1}^{k-1} \bp n\\ i\ep B^i.\]
We choose $n=(k-1)!m$. Then given any  $i\in \{1,\dots,k-1\}$ the natural number $m$ divides 
\[ \bp n\\ i \ep\,\,=\,\, \frac{n}{i}\bp n-1\\ i-1\ep.\]
This concludes the proof of the claim.

%

Now let $m\in \N$ and pick $n$ as in the previous claim. Since $A$ has infinite order the matrix $A^n-\id$ is non-zero. 
We deduce that $\Z^k/(A^n-\id)\Z^k$ contains an element of order $m$. 
\end{proof}

The following theorem is a slight generalization of results of Silver--Williams \cite{SW02} and Jean Raimbault ~\cite{Ra12}. 

\begin{theorem}\label{thm:sw02}
Let $F$ be a surface and $\psi\in \mod(F)$ such that the action of  $\psi$ on $H_1(S;\Z)$ has at least one eigenvalue $\lambda$ with $|\lambda|>1$, then for any $k\in\N$ there exists an $n\in \N$ such that $|\tor H_1(M(F;\psi^n);\Z)|>k$.
\end{theorem}

\begin{proof}
We denote by $p(t)$ the characteristic polynomial of the action of  $\psi$ on the free abelian group $H_1(S;\Z)$. 
Raimbault ~\cite[Theorem 0.2]{Ra12} showed that
\begin{align*}
	\lim\limits_{n\to\infty}\smfrac{1}{n}\cdot \log |\tor H_1(M(F;\psi^n);\Z)| \,\,=\,\,\log \Big(\,\ttmprod{i=0}{2g}\max\left\{1,|\l_i|\right\}\Big),
\end{align*}
where $\l_1,\ldots,\l_{2g}$ are the eigenvalues of $\psi$. By assumption $p(t)$ has a zero $\l$ with $|\l|>1$ and hence the right hand side is non-zero. But $\lim\limits_{n\to\infty}1/n=0$, so that $|\tor H_1(M(F;\psi^n);\Z)|$ cannot be bounded, i.e.\ it has to become arbitrarily large.
\end{proof}

\subsection{Proof of Theorems~\ref{mainthm} and~\ref{mainthm2}}

Before we give the proof of Theorem~\ref{mainthm} and Theorem~\ref{mainthm2} we recall the key theorem regarding surface-automorphisms and mapping tori. We refer to \cite{AFW15} for more details and further references.

\begin{theorem}\label{thm:nielsen-thurston}
	Let $F$ be a compact orientable surface with negative Euler characteristic, let $\psi\in\mod(F)$, and let $N=M(F,\psi)$. Then precisely one of the following three statements holds:
	\begin{enumerate}[font=\normalfont]
		\item  $\psi$ has finite order, in this case  $N$ is Seifert fibered.
		\item $\psi$ is pseudo-Anosov, in this case  $N$ is hyperbolic.
		\item  $\psi$ is reducible, i.e.\ after an isotopy there exists  a
		finite non-empty collection $\Gamma$ of disjoint essential simple closed curves in $F$  and a tubular neighborhood $\nu \Gamma$ such that $\Gamma$ and $\nu \Gamma$ are  fixed by $\psi$ setwise and such that the $\psi$-mapping tori of $\Gamma$ are precisely the JSJ-tori of  $N=M(F,\psi)$. 
	\end{enumerate}
\end{theorem}

\begin{proof}[Proof of Theorem~\ref{mainthm}]
Let $N$ be a prime 3-manifold that is not a graph manifold.
If $N$ is a closed hyperbolic 3-manifold then it follows from Sun's Theorem~\ref{thm:sun}  that  there exists a finite cover $\wti{N}$ of $N$ with $|\tor H_1(\wti{N};\Z)|> 0$. Henceforth we assume that $N$ is not a closed hyperbolic 3-manifold.

Since $N$ is not a graph manifold we know by the  virtual fibering theorem of Agol \cite{Ag08}, Przytycki--Wise \cite{PW12} and Wise \cite{Wi09,Wi12a,Wi12b}
(see also \cite[p.~94]{AFW15} for precise references and see \cite{FK14,CF17,GM17} for alternative proofs) that $N$ admits a finite cover that fibers over $S^1$. Therefore,  without loss of generality, we can assume that $N$ is already fibered. We denote by $F$ the fiber of a surface fiber structure of $N$ and we denote by $\psi$ the monodromy.

\begin{claim}
There
exists a finite  cover $\wti{F}\to F$ and a  lift $\wti{\phi}$ of some power $\psi^k$ such that the action of
$\wti{\phi}$ on $H_1(\wti{F};\Z)$  has infinite order.
\end{claim}

If $N$ has boundary, then the claim follows immediately from Theorem~\ref{thm:hadari}, in fact in this case one can arrange that $k=1$. 

Now suppose that $N$ is closed. 
By Theorem~\ref{thm:nielsen-thurston}  we can assume, possibly after an isotopy and possibly replacing  $\psi$ by a non-trivial power, that there exists a finite set of essential curves $c_1,\cdots,c_k$ such that the monodromy preserves the union of the $c_i$'s pointwise and such that it preserves the components of $F$ cut along the $c_i$'s setwise. Furthermore, since $N$ is not a graph manifold there exists a component $S$ of  $F$ cut along the $c_i$'s such that the action of $\psi$ on $S$ is pseudo-Anosov in particular has infinite order. 
Since we excluded the case that  $N$  is a closed hyperbolic manifold 
we see that  $\psi$ itself is not pseudo-Anosov. Therefore we know that $S$ is a proper subsurface of $F$. Since the $c_i$'s are essential we know that $\pi_1(S)\to \pi_1(F)$ is injective. 
It follows from Proposition~\ref{prop:hadari2} that there
exists a finite  cover $\wti{F}\to F$ and a  lift $\wti{\varphi}$ of some power $\psi^k$ such that the action of
$\wti{\varphi}$ on $H_1(\wti{F};\Z)$  has infinite order. This concludes the proof of the claim.

We first consider the case that all eigenvalues of $\wti{\varphi}_*\colon H_1(\wti{F};\Z)\to H_1(\wti{F};\Z)$ are roots of unity.
Then the desired statement follows from Proposition~\ref{prop:eigenvalue-root-of-unity} together with Lemma~\ref{lem:h1fibered}. If this is not the case, then it follows from Kronecker's Theorem \cite[Theorem~4.5.4]{Pr10} that there exists at least one eigenvalue $\lambda$ of $\wti{\varphi}_*$ with $|\lambda|>1$, but then the desired statement follows from Theorem~\ref{thm:sw02}  together with Lemma~\ref{lem:h1fibered}.
\end{proof}

The proof of Theorem~\ref{mainthm2} is very similar to the previous proof but we feel that the proofs are more readable if they are done separately.
In the proof we will make use of the following lemma.

\begin{lemma}\label{lem:at-least-two-jsj-tori}
Let $N$ be a closed irreducible 3-manifold with a non-trivial JSJ-decomposition that is not a torus bundle and that does not have a JSJ-piece that is a twisted $I$-bundle over the Klein bottle. Then there exists a finite cover $\wti{N}$ of $N$ with at least two JSJ-tori.
\end{lemma}

\begin{proof}
Clearly we only have to consider the case that $N$ has precisely one JSJ-torus.
 Since $\pi_1(N)$ is residually finite we can find an epimorphism $\pi_1(N)\to G$ onto a finite non-abelian group $G$. We denote by $p\colon \wti{N}\to N$ the corresponding finite covering.
Since $N$ is not a torus bundle and since no JSJ-piece of $N$ is a twisted $I$-bundle over the Klein bottle it follows from Proposition~1.9.2 and Theorem~1.9.3 of \cite{AFW15} that the JSJ-decomposition of $\wti{N}$ is the preimage of the JSJ-decomposition of $N$. Since $N$ has a non-trivial JSJ-decomposition and since $G$ is non-abelian we see that
the $p^{-1}(T)$ has at least two components, so by the above discussion  $\wti{N}$ has at least two JSJ-tori.
\end{proof}

\begin{proof}[Proof of Theorem~\ref{mainthm2}]
Let $N$ be a virtually fibered graph manifold, which is not a Seifert fibered space.
We may already assume that $N$ is fibered over a surface $F$ with monodromy $\psi$. Since $N$ is not a Seifert fibered space it follows from Theorem~\ref{thm:nielsen-thurston} that the monodromy $\psi\in \mod(F)$ has infinite order.
As in the proof of Theorem~\ref{mainthm} it suffices to prove the following claim.

\begin{claim}
There
exists a finite  cover $\wti{F}\to F$ and a  lift $\wti{\phi}$ of some power $\psi^k$ such that the action of
$\wti{\phi}$ on $H_1(\wti{F};\Z)$  has infinite order.
\end{claim}

If the fiber $F$ has non-empty boundary, then the claim follows from Theorem~\ref{thm:hadari}.

Therefore we will now assume that the fiber $F$ is a closed surface. 
If $F$ is a torus, then $\pi_1(F)=H_1(F;\Z)$, hence if $\psi$ has infinite order, then by the standard isomorphism $\mod(F)=\op{SL}(2,\Z)=\op{Aut}(H_1(F;\Z))$ the action on $H_1(F;\Z)$ also has infinite order. So we can in fact take $\wti{F}=F$.

Finally suppose that $F$ is a surface of genus greater equal to two.
By Lemma~\ref{lem:at-least-two-jsj-tori} we can, possibly after going to a finite cover, assume that $N$ has at least two JSJ-tori. It follows from Theorem~\ref{thm:nielsen-thurston} that there  essential curves $c_1,\cdots,c_k$ with $k\geq 2$ such that the monodromy preserves a tubular neighborhood of  $c_1\cup \dots\cup c_k$ setwise.
After going to a suitable power of the monodromy we can arrange that for each $i$ the monodromy preserves a tubular neighborhood $\nu c_i$  setwise. 
We define $S:=F\setminus\nu c_1$. Note that the mapping torus $M(S,\psi|_S)$ is a graph manifold with at least one JSJ-torus given by the mapping torus on the remaining $c_j$s. Therefore it follows again from
Theorem~\ref{thm:nielsen-thurston} that $\psi|_S$ has infinite order. Therefore we can apply Proposition~\ref{prop:hadari2}.
This concludes the proof of the claim.
\end{proof}

\printindex


\begin{thebibliography}{asd}



\bibitem[Ag08]{Ag08}
I. Agol, {\em Criteria for virtual fibering}, J. Topol. 1 (2008), no. 2, 269--284.

\bibitem[AFW15]{AFW15}
M. Aschenbrenner, S. Friedl and H. Wilton, {\em
3-manifold groups}, EMS lecture notes in mathematics (2015)

\bibitem[BV13]{BV13}
N. Bergeron and  A. Venkatesh, {\em
The asymptotic growth of torsion homology for arithmetic groups},
J. Inst. Math. Jussieu {12} (2013), no. 2, 391--447. 


\bibitem[CF17]{CF17}
D. Cooper and D. Futer, {\em Ubiquitous quasi-Fuchsian surfaces in cusped hyperbolic 3-manifolds}, Preprint (2017), arXiv:1705.02890.

\bibitem[FK14]{FK14}
S. Friedl and T. Kitayama, {\em
The virtual fibering theorem for 3-manifolds},
L'enseignement math\'ematique 60 (2014), 79--107. 

\bibitem[FW16]{FW16}
S. Friedl and H. Wilton, {\em The membership problem for 3-manifold groups is solvable},
Alg. Geom. Top. 16 (2016), 1827--1850. 


\bibitem[GM17]{GM17}
D. Groves and J. Manning, {\em Quasiconvexity and Dehn filling}, Preprint (2017), arXiv:1708.07968.

\bibitem[Ha15]{Ha15}
A. Hadari, {\em 
Every Infinite order mapping class has an infinite order action on the homology of some finite cover}, Preprint (2015)

\bibitem[KaM12]{KaM12}
J. Kahn and V. Markovic, {\em Immersing almost geodesic surfaces in a closed hyperbolic three manifold},
Ann. of Math. 175 (2012), 1127--1190.

\bibitem[Ko15]{Ko15}
T. Koberda, {\em Asymptotic linearity of the mapping class group and a homological version
of the Nielsen-Thurston classification},  Geom. Dedicata 156 (2012), 13--30.

\bibitem[KoM15]{KoM15}
T. Koberda, and J. Mangahas, {\em An effective algebraic detection of the Nielsen-Thurston
classification},  J. Topol. Analysis 7 (2015), 1--21.


\bibitem[Le14]{Le14}
T. Le, {\em Growth of  homology torsion in finite coverings and hyperbolic volume}, preprint, 
{\tt  arXiv:1412.7758v2} (2014), to appear in the Annales de l'Institut Fourier

\bibitem[Li17]{Li17}
Y. Liu, {\em Virtual homological spectral radii for automorphisms of surfaces}, preprint (2017), arXiv:1710.05039.


\bibitem[L\"u13]{Lu13} W. L\"uck, {\em Approximating $L^2$-invariants and homology growth}, Geom. Funct. Anal. {23} (2013), no. 2, 622--663.

\bibitem[Me16]{Me16}
J. Meumertzheim, {\em Growth of torsion in the homology of finite coverings of 3-manifolds}, Master Thesis, University of Regensburg (2016)

%

\bibitem[Pr10]{Pr10}
V. Prasolov, {\em Polynomials}, Algorithms and Computation in Mathematics, 11. Springer-Verlag, Berlin, 2010.

\bibitem[PW12]{PW12}
 P. Przytycki and D. Wise, {\em Mixed $3$-manifolds are virtually special}, preprint, 2012. To be published by the J. Amer. Math. Soc.

\bibitem[Ra12]{Ra12} J. Raimbault, {\em Exponential growth of torsion in abelian coverings}, Algebr. Geom. Topol. 12
(2012), 1331--1372.

\bibitem[Sc78]{Sc78}
P. Scott, {\em  Subgroups of surface groups are almost geometric},
J. Lond. Math. Soc. 17 (1978), 555--565.

\bibitem[Sc85]{Sc85}
P.  Scott. {\em Correction to ``Subgroups of surface group
s are almost geometric''}, J. Lond. Math. Soc. 32 (1985), 217--220.

\bibitem[SW02]{SW02}
D. Silver and S. Williams, {\em Mahler measure, links and homology growth},
Topology {41} (2002), no. 5, 979--991.


\bibitem[Su15]{Su15}
H. Sun, {\em Virtual homological torsion of closed hyperbolic 3-manifolds},
J. Differential Geom. 100 (2015), no. 3, 547--583. 



%

\bibitem[WY97]{WY97}
S. Wang and F. Yu, {\em Graph manifolds with non-empty boundary are covered by surface bundles}, Math. Proc. Cambridge Philos. Soc. 122 (1997), no. 3, 447--455.

\bibitem[Wi09]{Wi09}
D. Wise, {\em The structure of groups with a quasiconvex hierarchy}, Electronic Res. Ann. Math. Sci. 16 (2009), 44--55.

\bibitem[Wi12a]{Wi12a}
D. Wise, {\em The structure of groups with a quasi-convex hierarchy}, 189 pp., preprint, 2012,
downloaded on Oct. 29, 2012 from  
\mbox{\url{http://www.math.mcgill.ca/wise/papers.html}}

\bibitem[Wi12b]{Wi12b}
D. Wise, {\em From Riches to RAAGS: $3$-Manifolds, Right-angled Artin Groups, and Cubical Geometry,} 
CBMS Regional Conference Series in Mathematics, vol. 117, American Mathematical Society, Providence, RI, 2012.

\end{thebibliography}
\end{document}